\documentclass[11pt,letterpaper,reqno]{amsart}
\usepackage{tikz}
\usetikzlibrary{positioning, shapes.geometric, arrows.meta, calc, positioning}
\usepackage{amssymb,amsmath,amsthm,amsfonts}
\usepackage{bbm}
\usepackage{enumitem}
\usepackage{pgfplots}
\pgfplotsset{compat=1.18}
\usepackage{booktabs}
\usepackage{graphicx}
\usepackage[T1]{fontenc}
\usepackage{doi}
\usepackage{float}
\usepackage{bookmark}
\usepackage{hyperref}
\hypersetup{pdfstartview={FitH}}

\newcommand{\C}{\mathbb{C}}

\newtheorem{thm}{Theorem}[section]

\newtheorem{cor}[thm]{Corollary}

\newtheorem{conj}[thm]{Conjecture}
\theoremstyle{definition}

\newtheorem{remark}[thm]{Remark}

\numberwithin{equation}{section}

\title[Dual Smale for odd polynomials]{Dual Smale's mean value conjecture for odd polynomials}
\author[Q.~Tang]{Quanyu Tang}
\address{School of Mathematics and Statistics, Xi'an Jiaotong University, Xi'an 710049, P. R. China}
\email{tang\_quanyu@163.com}


\begin{document}

\begin{abstract}
We prove Dual Smale's mean value conjecture for all odd polynomials with nonzero linear term. 
Precisely, if $P$ is an odd polynomial of degree $d\ge3$ with $P(0)=0$ and $P'(0)=1$, then there exists a critical point $\zeta$ of $P$ such that 
$$
\left|\frac{P(\zeta)}{\zeta}\right| \ge \frac1d.
$$This result can be regarded as a dual counterpart of T.~W.~Ng's theorem on Smale's mean value conjecture for odd polynomials with nonzero linear term [J. Aust. Math. Soc. 75 (2003), 409--411].
\end{abstract}

\maketitle

\section{Introduction}
Throughout this paper, we consider complex polynomials of the form
\[
P(z)=c_0+c_1z+\cdots+c_nz^n,\qquad c_n\ne0.
\]
A point $\zeta\in\C$ is called a \emph{critical point} of $P$ if $P'(\zeta)=0$, 
and the corresponding value $P(\zeta)$ is called a \emph{critical value} of $P$.

In his celebrated 1981 paper \cite{Smale81}, Smale proved that for every complex polynomial $P$ of degree $n\ge2$ and every noncritical point $z$ there is a critical point $\zeta$ with
\begin{equation}\label{eq:smale-4}
\left|\frac{P(z)-P(\zeta)}{z-\zeta}\right| \le 4 |P'(z)|,
\end{equation}
and conjectured that $4$ can be replaced by $1$ (indeed $1-1/n$ is best possible). It is easy (see~\cite{BMN02}) to show that Smale's conjecture is equivalent to the following normalized  conjecture:
\begin{conj}[Smale's mean value conjecture]
Let $P(z)$ be a complex polynomial of degree $n\ge2$ satisfying $P(0)=0$ and $P'(0)=1$. 
Then there exists a critical point $\zeta$ of $P$ such that 
\[
\left|\frac{P(\zeta)}{\zeta}\right|\le 1-\frac1n.
\]
\end{conj}

The conjecture has resisted proof despite extensive efforts by many leading mathematicians. Partial estimates have been obtained for several special families of polynomials (see \cite{BMN02, HK10}), yet the general case remains unresolved. 

About three decades later, a dual version of Smale's mean value conjecture was independently proposed by Dubinin--Sugawa~\cite{DubininSugawa09} and Ng:

\begin{conj}[Dual Smale's mean value conjecture]
For every complex polynomial $P(z)$ of degree $n\ge2$ such that $P(0)=0$ and $P'(0)=1$, there exists a critical point $\zeta$ of $P$ such that 
\[
\left|\frac{P(\zeta)}{\zeta}\right|\ge \frac1n.
\]
\end{conj}

The Dual Smale's conjecture is known to be true for $n\le7$. 
The cases $2\le n\le4$ were proved in~\cite{DubininSugawa09, Tischler89}. The cases $n=5$ and $n=6$ were proved by Hinkkanen, Kayumov, and Khammatova~\cite{HKK2019}
using the idea of selecting a critical point of minimal modulus; the same approach extends to $n=7$, for which a rigorous proof was given later in~\cite{HKK2025}. Recently, Kayumov and Khammatova~\cite{KKhJMAA} proposed a new conjectured inequality implying the Dual Smale's conjecture and verified it for all polynomials whose critical points lie on a single ray. Earlier, Tischler~\cite{Tischler89} established the conjecture for the special class of \emph{conservative polynomials}, characterized by \(
P'(\zeta)=0 \Rightarrow P(\zeta)/\zeta=\mathrm{const}
\). As for general bounds, Dubinin and Sugawa~\cite{DubininSugawa09} first proved a weakened form with constant $\frac{1}{n4^{n}}$; later, Ng and Zhang~\cite{NgZhang16} improved the exponential bound to $\frac{1}{4^{n}}$. A polynomial-order lower bound was then obtained by Dubinin~\cite{DubininSurvey,Dubinin19}:

\begin{thm}\label{thm:dubinin}
Let $P$ be a complex polynomial of degree $n\ge2$. Then for every $z\in\C$ there exists a critical point $\zeta$ of $P$ such that
\[
\left|\frac{P(z)-P(\zeta)}{z-\zeta}\right|\ge 
\frac{1}{n}\tan\!\frac{\pi}{4n} |P'(z)|.
\]
Moreover, by a more complicated technique (see \cite[Theorem~3.9 and the remark to it]{DubininSurvey}), the factor $\tan(\pi/(4n))$ can be replaced by $1/n$; equivalently,
\begin{equation}\label{eq:dubinin-1n2}
\left|\frac{P(z)-P(\zeta)}{z-\zeta}\right|\ge \frac{1}{n^2} |P'(z)|.
\end{equation}
\end{thm}

Inequality~\eqref{eq:dubinin-1n2} provides the best known universal lower bound for general complex polynomials.  

Our goal in this paper is to prove that the Dual Smale's mean value conjecture holds for all odd polynomials with nonzero linear term.
This family was previously investigated by Ng~\cite{NgOdd} in connection with the classical Smale's mean value conjecture, where he established the corresponding inequality with constant~$2$. Here we show that, in the dual setting, the conjectured constant~$1/n$ indeed holds for all such polynomials.


\begin{thm}\label{thm:main}
Let $P$ be an odd polynomial of degree $d\ge3$ with $P(0)=0$ and $P'(0)=1$. Then
\[
\max_{P'(\zeta)=0}\left|\frac{P(\zeta)}{\zeta}\right| \ge \frac1d.
\]
\end{thm}

\begin{proof}
Write $P(z)=z Q(z^2)$; here $Q$ is a polynomial of degree $\frac{d-1}{2}$ with $Q(0)=P'(0)=1$.
Set
\begin{equation}\label{eq:H-and-R}
H(u):=u Q(u)^2,\qquad 
R(u):=Q(u)+2uQ'(u).
\end{equation}
Then
\begin{equation}\label{eq:Hprime}
H'(u)=Q(u)\left(Q(u)+2uQ'(u)\right)=Q(u)R(u),
\end{equation}
so the critical points of $H$ are exactly $\{Q=0\}\cup\{R=0\}$. Note also that
\begin{equation}\label{eq:H-at-0}
H(0)=0,\qquad H'(0)=Q(0)^2=1.
\end{equation}

By inequality~\eqref{eq:dubinin-1n2} from Theorem~\ref{thm:dubinin}: for any polynomial $F$ of degree $n$ and any $a\in\C$ there exists a critical point $b$ of $F$ such that
\begin{equation}\label{eq:Dub-1overn2}
\left|\frac{F(a)-F(b)}{a-b}\right| \ge \frac{1}{n^2} |F'(a)|.
\end{equation}
Apply inequality~\eqref{eq:Dub-1overn2} to $F=H$ with $n=d$ and $a=0$. Using \eqref{eq:H-at-0}, we obtain a critical point $w$ of $H$ such that
\begin{equation}\label{eq:apply-to-H}
\left|\frac{H(w)}{w}\right| \ge \frac{1}{d^2} |H'(0)|=\frac{1}{d^2}.
\end{equation}
Since the right-hand side is positive, $H(w)\neq0$; hence $w\notin\{Q=0\}$ and by \eqref{eq:Hprime} necessarily $w\in\{R=0\}$. But then $H(w)=w\,Q(w)^2$, and \eqref{eq:apply-to-H} gives
\[
|Q(w)| =  \sqrt{\left|\frac{H(w)}{w}\right|}  \ge  \frac{1}{d}.
\]
Choose $c\in\{\pm\sqrt{w}\}$. From $R(w)=0$ and \eqref{eq:H-and-R} we have
\[
P'(c)=Q(w)+2wQ'(w)=0,\qquad 
\left|\frac{P(c)}{c}\right|=\left|Q(w)\right| \ge \frac{1}{d}.
\]
Hence $\max_{P'(\zeta)=0}\left|P(\zeta)/\zeta\right|\ge 1/d$, completing the proof. \end{proof}

\begin{remark}
We point out that Theorem~\ref{thm:main} cannot be attained with equality. 
Indeed, assuming the full Dual Smale's mean value conjecture holds for all degree-$d$ polynomials, 
one can repeat the argument of Theorem~\ref{thm:main} for the auxiliary function $H(u)=uQ(u)^2$ to obtain $\max_{P'(\zeta)=0}|P(\zeta)/\zeta|\ge 1/\sqrt{d}$. This conditional strengthening implies that global equality 
$\max_{P'(\zeta)=0}|P(\zeta)/\zeta|=1/d$ cannot occur in the odd subclass for any $d\ge3$.
\end{remark}

In the spirit of Ng~\cite[Note added in proof]{NgOdd}, the following dual variant holds for polynomials $P$ satisfying $P(\lambda z)\equiv \lambda P(z)$, where $\lambda$ is a primitive $k$th root of unity.

\begin{cor}\label{cor:k}
Let $P$ be a polynomial of degree $d\ge2$ with $P(0)=0$ and $P'(0)=1$. 
Assume there exists a primitive $k$th root of unity $\lambda$ such that $P(\lambda z)\equiv \lambda\,P(z)$. Then
\[
\max_{P'(\zeta)=0}\left|\frac{P(\zeta)}{\zeta}\right| \ge \frac{1}{d^{2/k}}.
\]
\end{cor}
\begin{proof}
The proof is identical to that of Theorem~\ref{thm:main}, 
except that $P(z)=zQ(z^k)$ and $H(u)=uQ(u)^k$ replace $P(z)=zQ(z^2)$ and $H(u)=uQ(u)^2$. 
Applying inequality~\eqref{eq:Dub-1overn2} to $H$ yields 
a critical point $\zeta$ with $|P(\zeta)/\zeta|\ge d^{-2/k}$.
\end{proof}
Therefore every polynomial satisfying the assumptions of Corollary~\ref{cor:k} 
indeed satisfies the Dual Smale's mean value conjecture.


\section*{Acknowledgements}
The author would like to thank Prof.~Minghua Lin for helpful comments and suggestions.

\end{document}